\documentclass[11pt,francais]{article}
\usepackage{amsmath}\usepackage{epsf,amsfonts,amsthm}\usepackage{amscd,amssymb}
\usepackage{xcolor,epic,eepic}\usepackage{epsfig}
\usepackage{fontenc,indentfirst, delarray,amsfonts,amsmath,amssymb}
\usepackage{rotating}
\usepackage{mathdots}
\usepackage[T1]{fontenc}
\usepackage[matrix,arrow,curve]{xy}
\usepackage{amsmath}
\usepackage{amssymb}
\usepackage{amsthm}
\usepackage{amscd}
\usepackage{amsfonts}
\usepackage{graphicx}
\usepackage{fancyhdr}
\usepackage{dsfont,texdraw}
\usepackage{amsmath}\usepackage{epsf,amsfonts,amsthm}\usepackage{amscd,amssymb}
\usepackage{xcolor,epic,eepic}\usepackage{epsfig}
\usepackage{fontenc,indentfirst, delarray,amsfonts,amsmath,amssymb}
\usepackage{rotating}
\usepackage{amscd}
\usepackage{amsmath}
\usepackage{amsthm}
\usepackage{mathrsfs}
\usepackage{latexsym}
\usepackage{amssymb}
\usepackage{amsfonts}
\theoremstyle{plain}

\usepackage{tikz}
\usetikzlibrary{arrows,chains,matrix,positioning,scopes,snakes}
\makeatletter
\tikzset{join/.code=\tikzset{after node path={%
\ifx\tikzchainprevious\pgfutil@empty\else(\tikzchainprevious)%
edge[every join]#1(\tikzchaincurrent)\fi}}}
\makeatother

\tikzset{>=stealth',every on chain/.append style={join},
         every join/.style={->}}
\tikzset{
    >=stealth',
    punkt/.style={
           rectangle,
           rounded corners,
           draw=black, very thick,
           text width=6.5em,
           minimum height=2em,
           text centered},
    pil/.style={
           ->,
           thick,
           shorten <=2pt,
           shorten >=2pt,}
}

\usepackage[T1]{fontenc}
\newcommand{\bee}{\begin{enumerate}}
\newcommand{\eee}{\end{enumerate}}
\newcommand{\benn}{\begin{equation*}}
\newcommand{\eenn}{\end{equation*}}
\newcommand{\be}{\begin{equation}}
\newcommand{\ee}{\end{equation}}
\newcommand{\bean}{\begin{eqnarray}}
\newcommand{\eean}{\end{eqnarray}}
\newcommand{\bea}{\begin{eqnarray*}}
\newcommand{\eea}{\end{eqnarray*}}
\newcommand{\w}{\wedge}
\newcommand{\ot}{\otimes}

\newcommand{\p}{\partial}
\newcommand{\la}{\langle}
\newcommand{\ra}{\rangle}

\newcommand{\Ci}{C^{\infty}}

\newcommand{\R}{\mathbb{R}}

\newcommand{\op}[1]{\!\!\mathop{\rm ~#1}\nolimits}

\mathchardef\za="710B  
\mathchardef\zb="710C  
\mathchardef\zg="710D  
\mathchardef\zd="710E  
\mathchardef\zve="710F 
\mathchardef\zz="7110  
\mathchardef\zh="7111  
\mathchardef\zy="7112 

\mathchardef\zi="7113  
\mathchardef\zk="7114  
\mathchardef\zl="7115  
\mathchardef\zm="7116  
\mathchardef\zn="7117  
\mathchardef\zx="7118  
\mathchardef\zp="7119  
\mathchardef\zr="711A  
\mathchardef\zs="711B  
\mathchardef\zt="711C  
\mathchardef\zu="711D  
\mathchardef\zf="711E 
\mathchardef\zq="711F  
\mathchardef\zc="7120  
\mathchardef\zw="7121  
\mathchardef\ze="7122  
\mathchardef\zvy="7123  
\mathchardef\zvw="7124  
\mathchardef\zvr="7125 
\mathchardef\zvs="7126 
\mathchardef\zvf="7127  
\mathchardef\zG="7000  
\mathchardef\zD="7001  
\mathchardef\zY="7002  
\mathchardef\zL="7003  
\mathchardef\zX="7004  
\mathchardef\zP="7005  
\mathchardef\zS="7006  
\mathchardef\zU="7007  
\mathchardef\zF="7008  
\mathchardef\zW="700A  

\newcommand{\cyclic}{\mathop{\kern0.9ex{{+}
\kern-2.15ex\raise-.25ex\hbox{\Large\hbox{$\circlearrowright$}}}}\limits}

\newcommand{\cE}{{\cal E}}

 \newcommand{\cC}{{\cal C}}
 \newcommand{\cA}{{\cal A}}
 \newcommand{\cM}{{\cal M}}

 \newcommand{\cI}{{\cal I}}

\newtheorem{rem}{Remark}
\newtheorem{theo}{Theorem}
\newtheorem{prop}{Proposition}
\newtheorem{lem}{Lemma}

\newtheorem{ex}{Example}
\newtheorem{defi}{Definition}

\DeclareMathAlphabet{\mathpzc}{OT1}{pzc}{m}{it}

\newcommand{\Hom}{\mathrm{Hom}}

   \def\i{\imath}

\usepackage{amsthm,amsmath,amssymb,latexsym,amscd}

\newcommand{\Der}{\mathrm{Der}}

\renewcommand{\c}{\circ}

\newtheorem{definition}{Definition}[section]

\newtheorem{proposition}[definition]{Proposition}

\date{}

\begin{document}
\title{The Free Courant Algebroid}
\author{Beno\^it Jubin, Norbert Poncin, and Kyousuke Uchino}
\maketitle
\abstract
{We introduce the category of generalized Courant algebroids and show that it admits a free object on any anchored vector bundle. The free Courant algebroid is built from two components: the generalized Courant algebroid associated to a symmetric Leibniz algebroid and the free symmetric Leibniz algebroid on an anchored vector bundle. Our construction is thus based on the new concept of symmetric Leibniz algebroid. We compare this subclass of Leibniz algebroids with the subclass of Loday algebroids that was introduced and studied in \cite{GKP}.}

\vspace{5mm} \noindent {\bf MSC 2010}: 17B01; 17B62; 53D17; 22A22\medskip


\thispagestyle{empty}
\tableofcontents

\section{Introduction}

The skew-symmetric non-Jacobi Courant bracket \cite{Co} on sections of ${\cal T}M:= TM\oplus T^\ast M$ was originally introduced by Courant to formulate the integrability condition defining a Dirac structure. However its nature became clear only due to the observation by Liu, Weinstein and Xu \cite{LWX} that $\cal T M$ endowed with the Courant bracket plays the role of a `double' object, in the sense of Drinfeld \cite{Dr}, for a pair of Lie algebroids over $M$. Whereas any Lie bialgebra has a double which is a Lie algebra, the double of a Lie bialgebroid is not a Lie algebroid, but a Courant algebroid -- a generalization of ${\cal T} M$ equipped with the Courant bracket. There is another way of viewing {\it Courant algebroids} as a generalization of Lie algebroids. This requires a change in the definition of the Courant bracket and the use of an analog of the non-antisymmetric Dorfman bracket \cite{Do}. The traditional Courant bracket then becomes the skew-symmetrization of the new one \cite{Roy}. This change replaces one defect with another: a version of the Jacobi identity is satisfied, while the bracket is no longer skew-symmetric. Such algebraic structures have been introduced by Loday \cite{Lo} under the name of {Leibniz algebras}. Canonical examples of Leibniz algebras arise often as {\em derived brackets} introduced by Kosmann-Schwarzbach \cite{K-S,YKS}. Since Leibniz brackets appear naturally in Geometry and Physics in the form of `algebroid brackets', i.e. brackets on sections of vector bundles, there were a number of attempts to formalize
the concept of {\em Leibniz algebroid} \cite{Ba,BV,G1,GKP,GM,ILMP,Ha,HM,KS,MM,SX,Wa}. Note also that a Leibniz algebroid is the horizontal categorification of a Leibniz algebra; vertical categorification leads to Leibniz $n$-algebras and Leibniz $n$-algebroids \cite{AP10,KMP11,KPQ,DP,BP12}.\medskip

In this article, we introduce the category of generalized Courant algebroids and show that it admits a free object on any anchored vector bundle. Our construction is based on the new concept of symmetric Leibniz algebroid. We compare this subclass of Leibniz algebroids with the subclass of Loday algebroids that was introduced and studied in~\cite{GKP}.
\medskip

The present paper is organized as follows. In Section 2, we recall the definitions of the categories of Leibniz algebroids \cite{ILMP}, of Leibniz pseudoalgebras -- their algebraic counterpart --, and of modules over them, as well as the classical notion of Courant algebroid. We then describe, in Section 3, two intersecting subclasses of Leibniz algebroids, namely the class of Loday algebroids (Definition \ref{DLodAld}) \cite{GKP}, which are Leibniz algebroids that admit a generalized right anchor and are thus geometric objects, and the class of symmetric Leibniz algebroids (Definition \ref{DSymLeiAld}), a new concept, made of Leibniz algebroids that satisfy two weak differentiability conditions on the left argument and contain Courant algebroids as a particular example. Examples of symmetric and nonsymmetric Leibniz and Loday algebroids are given. In Section 4, we motivate the definition of generalized Courant algebroids (Definitions \ref{DPrCrtAld} and \ref{DCrtAld}), which are specific symmetric Leibniz algebroids. The prototypical example of a generalized Courant algebroid is the one naturally associated to a symmetric Leibniz algebroid (Theorem \ref{Main}). This associated Courant algebroid is one of the two ingredients of the free Courant algebroid. Moreover, Theorem \ref{Main} allows to understand the origin of the definition of symmetric Leibniz algebroids. The second ingredient is the free symmetric Leibniz algebroid, which we construct in Section 5 (Theorem \ref{Main2} and Proposition \ref{Main3}). In Section 6, we combine the results of Section 4 and Section 5 to build the free Courant algebroid (Theorem \ref{Main4}).\medskip

\newcommand{\inser}[1]{\textcolor{blue}{#1}}
\newcommand{\suppr}[1]{\textcolor{red}{\sout{#1}}}
\newcommand{\define}[1]{\emph{#1}}

\newcommand{\supprim}[1]{\textcolor{red}{\sout{#1}}}

\newcommand{\frakX}{\mathfrak{X}}
\newcommand{\cat}[1]{\texttt{#1}}

\section{Preliminaries}

\subsection{Notation and Conventions}

Unless otherwise specified, manifolds are made of a finite-dimensional smooth structure on a second-countable Hausdorff space.

If $[-,-]$ is a Leibniz bracket, we denote by $-\circ-$ its symmetrization, that is,
\begin{equation}
X \circ Y := [X,Y] + [Y,X]\;
\end{equation}
for any elements $X, Y$ of the Leibniz algebra.

\subsection{Anchored Vector Bundles and Anchored Modules}\label{subsec:anch}
\begin{defi}
If $M$ is a manifold, an \define{anchored vector bundle} over $M$ is a vector bundle $E \to M$ with a vector bundle morphism $a \colon E \to TM$, called its \define{anchor}.

If $R$ is a commutative unital ring and $\cA$ is a commutative unital $R$-algebra, an \define{anchored module} over $\cA$ is an $\cA$-module $\cE$ with an $\cA$-module morphism $a \colon \cE \to \Der \cA$, called its \define{derivative}.
\end{defi}

Of course, here $TM$ is the tangent bundle of $M$ and $\Der\cA$ is the module of derivations of $\cA.$ If $a \colon E \to TM$ is an anchor, we still denote by $a \colon \Gamma E \to \frakX(M) := \Der(\Ci(M))$ the corresponding operation on sections defined pointwise.
Obviously, if $E$ is an anchored vector bundle over $M$ with anchor $a$, then its space $\Gamma E$ of sections is an anchored module over $(\R,\Ci(M))$ with derivative $a$.\medskip

Morphisms of anchored vector bundles (resp., anchored modules) {over a fixed base} (resp., {over a fixed algebra}) are defined in an obvious way, and we obtain categories $\cat{AncVec}(M)$ and $\cat{AncMod}(\cA)$, respectively.
The algebroids (resp., pseudoalgebras) we are going to define in this article will be anchored vector bundles (resp., anchored modules) with extra structure.
They will form, together with their morphisms, categories that are concrete over ${\tt AncVec}(M)$ and ${\tt AncMod}(\cA)$, that is, admit a (faithful) forgetful functor to the latter.
One of our goals is to define left adjoints to these functors, or in other words, to define the free algebroid (resp., pseudoalgebra) of a given type on a given anchored vector bundle (resp., anchored module).
These constructions can be seen as the horizontal categorification of the free Lie, Leibniz$\ldots$ algebra on a given vector space.\medskip

For the different types of pseudoalgebras we are going to define, we will also define \define{modules} over them, using the following general principle: if $V$ is an $R$-module with extra structure, then a \define{$V$-module} is an $R$-module $W$ such that $V \oplus W$ is of the same type as $V$ and contains $V$ as a subobject and $W$ as an abelian ideal in an appropriate sense.
Similarly, a morphism of modules from the $V$-module $W$ to the $V'$-module $W'$ will be a morphism $V \oplus W \to V' \oplus W'$ sending $V$ to $V'$ and $W$ to $W'$.
It is possible to make these statements precise, but we prefer to keep them heuristic here and to work out the details below in the case of Leibniz modules.

\subsection{Leibniz Algebroids}

In this paper, we consider \emph{left Leibniz brackets}, i.e. bilinear brackets that satisfy the Jacobi identity
\begin{equation}\label{eq:jacobi}
[X,[Y,Z]]=[[X,Y],Z]+[Y,[X,Z]]\;.
\end{equation}

We first recall the definition of a Leibniz algebroid given in \cite{ILMP}. Note that this notion of Leibniz algebroid does not impose any differentiability requirement on the first argument of the bracket and is thus not a geometric concept.

\begin{defi}\label{LeiOid} A {\em Leibniz algebroid} is an anchored vector bundle $E\to M$ together with a Leibniz bracket $[-,-]$ on its space $\zG E$ of sections, which satisfy
\begin{equation}\label{axiom1}
[X,fY]=f[X,Y]+a(X)(f)Y,
\end{equation}
for any $f\in C^{\infty}(M)$ and $X,Y\in\Gamma E$.
\end{defi}
It is easily checked that the Leibniz rule (\ref{axiom1}) and the Jacobi identity imply that $a$ is a Leibniz algebra morphism: \begin{equation}\label{axtiom2}
a[X,Y]=[a(X),a(Y)],
\end{equation}
where the {\small RHS} bracket is the Lie bracket on $\Gamma TM$.

Let us also mention that here and in the following we consider {\it left} Leibniz brackets, i.e. bilinear brackets that satisfy the Jacobi identity $$[X,[Y,Z]]=[[X,Y],Z]+[Y,[X,Z]]\;.$$

We will essentially deal with the algebraic counterpart of Leibniz algebroids:

\begin{defi}\label{LeiPsAlg} Let $R$ be commutative unital ring and let $\cA$ be a commutative unital $R$-algebra. A {\em Leibniz pseudoalgebra} (or {\em Leibniz-Rinehart algebra}) over $(R,\cA)$ is an anchored module $(\cE,a)$ over $\cA$ endowed with a Leibniz $R$-algebra structure $[-,-]$, such that, for all $f\in\cA$ and $X,Y\in\cE,$ \begin{itemize}\item $[X,fY]=f[X,Y]+a(X)(f)\,Y$ and \item $a[X,Y]=[a(X),a(Y)]$, where the {\small RHS} is the commutator.\end{itemize} \end{defi}

If the $\cA$-module $\cE$ is faithful, the last requirement is again a consequence of the Leibniz rule and the Jacobi identity.\medskip

The space of sections of a Leibniz algebroid over $M$ is obviously a Leibniz pseudoalgebra over the $\R$-algebra $\Ci(M)$.\medskip

Of course, if, in Definitions \ref{LeiOid} and \ref{LeiPsAlg}, the Leibniz bracket is antisymmetric, we get a Lie algebroid and a Lie pseudoalgebra, respectively. \medskip

Leibniz algebroids over $M$ and Leibniz pseudoalgebras over $(R,\cA)$ are the objects of categories ${\tt LeiOid}\,M$ and ${\tt LeiPsAlg}\,(R,\cA)$, respectively. The morphisms of these categories are defined as follows.

\begin{defi}
Let $(E_{1},[-,-]_1,a_1)$ and $(E_{2},[-,-]_2,a_2)$ be two Leibniz algebroids over a same manifold $M$. A {\em Leibniz algebroid morphism} between them is a bundle map $\phi:E_{1}\to E_{2}$ such that $a_{2}\,\phi=a_{1}$ and  $\phi[X,Y]_1=[\phi X,\phi Y]_2$, for any $X,Y\in\Gamma E_{1}$.
\end{defi}

\begin{defi}
Let $(\cE_{1},[-,-]_1,a_1)$ and $(\cE_{2},[-,-]_2,a_2)$ be two Leibniz pseudoalgebras over the same pair $(R,\cA)$. A {\em Leibniz pseudoalgebra morphism} between them is an $\cA$-module morphism $\phi:\cE_{1}\to \cE_{2}$ such that $a_{2}\,\phi=a_{1}$ and  $\phi[X,Y]_1=[\phi X,\phi Y]_2$, for any $X,Y\in\cE_{1}$.
\end{defi}

We now define (bi)modules over Leibniz algeboids and pseudoalgebras.\medskip

Recall first the definition of a module over a Leibniz $R$-algebra $(V,[-,-])$.
By the general heuristic described above, this is an $R$-module $W$ with a Leibniz $R$-algebra structure $[-,-]$ on $V \oplus W$ containing $V$ as a subalgebra and $W$ as an abelian ideal.
Therefore, this bracket has to be the original bracket on $V \times V$, and 0 on $W \times W$, so it is determined by the values of $[x,w]$ and $[w,x]$, where $x \in V$ and $w \in W$.
Setting $\mu^l(x)(w) = [x,w]$ and $\mu^r(x)(w) = [w,x]$, we recover the usual notion: A {\it (bi)module} over a Leibniz $R$-algebra $(V,[-,-])$ is an $R$-module $W$ together with a left and a right {\it action} $\zm^l\in\op{Hom}_R(V\otimes_R W,W)$ and $\zm^r\in\op{Hom}_R(W\otimes_R V,W)$, which satisfy the following requirements \be\label{VVW}
\zm^r[x,y]=\zm^r(y)\zm^r(x)+\zm^l(x)\zm^r(y)\;,\ee \be\label{WVV}
\zm^r[x,y]=\zm^l(x)\zm^r(y)-\zm^r(y)\zm^l(x)\;,\ee \be\label{VWV}
\zm^l[x,y]=\zm^l(x)\zm^l(y)-\zm^l(y)\zm^l(x)\;,\ee for all $x,y\in V.$

In particular, let $\nabla$ be a representation of $(V,[-,-])$ on $W$, i.e. a Leibniz $R$-algebra morphism $V\to \op{End}_R(W)$. Then $\zm^\ell=\nabla$ and $\zm^r=-\nabla$ is a module structure over $V$ on $W$. \medskip

\begin{defi}\label{algebroidmodule}
Let $(E_{1},[-,-],a)$ be a Leibniz algebroid over $M$. A {\em module over the Leibniz algebroid} $E_1$ is a $C^{\infty}(M)$-module $\cE_2$ (not necessarily a locally free sheaf of $\Ci$-modules, i.e. not necessarily a module of sections of a vector bundle), which is a module $(\zm^\ell,\zm^r)$ over the Leibniz $\R$-algebra $\Gamma E_{1}$ whose left action satisfies the Leibniz rule
\be\label{LeibRule}
\zm^\ell(X)(fY)=f\zm^\ell(X)(Y)+a(X)(f)Y\;,
\ee
for any $f\in\Ci(M), X\in\Gamma E_{1}$, and $Y\in \cE_{2}$ (for convenience, we denote the left $\zG E_1$-action on $\cE_2$ and the Leibniz bracket on $\zG E_1$ by the same symbol).\end{defi}

\begin{rem}
In this definition, the $\Ci(M)$-module $\cE_2$ is not required to define a locally free sheaf of $\Ci(M)$-modules, i.e. it is not required to be a module of sections of a vector bundle.
\end{rem}

Similarly,

\begin{defi}\label{pseudoalgebramodule}
Let $(\cE_{1},[-,-],a)$ be a Leibniz pseudoalgebra over $(R,\cA)$. A {\em module over the Leibniz pseudoalgebra} $\cE_1$ is an $\cA$-module $\cE_2$ (hence an $R$-module), which is a module over the Leibniz $R$-algebra $\cE_{1}$ whose left action $\zm^\ell$ satisfies the Leibniz rule
$$
\zm^\ell(X)(fY)=f\zm^\ell(X)(Y)+a(X)(f)Y\;,
$$
for any $f\in\cA, X\in\cE_{1}$, and $Y\in \cE_{2}$.\end{defi}

In the case $\cE_1=\zG E_1$ and $\cE_2=\zG E_2$, where $E_1$  and $E_2$ are vector bundles over $M$, and $(\zm^\ell,\zm^r)=(\nabla,-\nabla)$, where $\nabla$ is a representation of $\zG E_1$ on $\zG E_2$, we deal with an $\R$-bilinear map $\nabla:\zG E_1\times \zG E_2\to \zG E_2$ such that, for any $f\in\Ci(M)$, $X,X_1,X_2\in\zG E_1$, and $Y\in\zG E_2$, $$\nabla_{[X_1,X_2]}=[\nabla_{X_1},\nabla_{X_2}]$$ and $$\nabla_X(fY)=a(X)(f)Y+f\nabla_XY\;.$$ If $\nabla$ is in addition $\Ci(M)$-linear in its first argument, the module structure is nothing but a flat $E_1$-connection on $E_2$. In the case of a Lie algebroid $E_1$, we thus recover the classical concept of $E_1$-module.

Note also that for any Leibniz algebroid $(E_1,[-,-],a)$, the $\Ci(M)$-module $\cE_2=\zG(M\times\R)=\Ci(M)$ and the actions $(\zm^\ell,\zm^r)=(a,-a)$ define on $\Ci(M)$ a module structure over the Leibniz algebroid $E_1$.

Let us emphasize that, just as for the Leibniz bracket on $\zG E_1$, we do not impose any differentiability condition on $\zm^r$.\medskip

Finally we define morphisms of modules over Leibniz pseudoalgebras.
From the above heuristic, a morphism from the $\cE_1$-module $\cE_2$ to the $\cE_1'$-module $\cE_2'$ should be a Leibniz pseudoalgebra morphism from $\cE_1 \oplus \cE_2$ to $\cE_1' \oplus \cE_2'$ sending $\cE_1$ to $\cE_1'$ and $\cE_2$ to $\cE_2'$. Unpacking this principle gives the following

\begin{defi} Let $(\cE_1,[-,-],a)$ and $(\cE'_1,[-,-]',a')$ be two Leibniz pseudoalgebras, and let $(\cE_2,\zm^\ell,\zm^r)$ and $(\cE'_2,\zm'^\ell,\zm'^r)$ be an $\cE_1$-module and an $\cE'_1$-module, respectively. A {\em morphism} between these two modules, is a pair $(\zf_1,\zf_2)$ made of a morphism $\zf_1:\cE_1\to\cE'_1$ of Leibniz pseudoalgebras and an $\cA$-linear map $\zf_2:\cE_2\to\cE'_2$, such that \be\label{RespActs}\zm'^\ell(\zf_1\times\zf_2)=\zf_2\,\zm^\ell\;\;\text{and}\;\;\zm'^r(\zf_2\times\zf_1)=\zf_2\,\zm^r\;.\ee\end{defi}

\subsection{Courant Algebroids}

As for the definition of Courant algebroids, we refer the reader to \cite{LWX}, \cite{Roy0}, \cite{GM2}, \cite{Uch}, and \cite{Kos3}.

\begin{defi}\label{CrtAld} A {\em Courant algebroid} is an anchored vector bundle $E\to M$, with anchor $a$, together with a Leibniz bracket $[-,-]$ on $\zG E$ and a bundle map $(-|-) \colon E\otimes E\to M\times\R$ that is in each fiber nondegenerate symmetric, called \define{scalar product}, which satisfy \be\label{4} a(X)(Y|Y)=2(X|[Y,Y])\;,\ee \be a(X)(Y|Y)=2([X,Y]|Y)\;,\label{5} \ee for any $X,Y\in\zG E$.
\end{defi}

\medskip\noindent The nondegeneracy of the scalar product allows us to identify $E$ with its dual $E^*$, and we will use this identification implicitly in the following. Note that (\ref{4}) is equivalent to \begin{equation}\label{4a} a(X)(Y|Z)=(X|Y\circ Z)\;,\end{equation} where $Y\circ Z$ denotes the symmetrized bracket. Similarly, (\ref{5}) easily implies the invariance of the scalar product, \begin{equation}\label{6}a(X)(Y|Z)=([X,Y]|Z)+(Y|[X,Z])\;,\end{equation} which in turn shows that $a$ is the anchor of the left adjoint map: \be\label{zr} [X,fY]=f[X,Y]+a(X)(f)Y\;. \ee Hence, a Courant algebroid is a particular Leibniz algebroid. When defining a derivation $D:\Ci(M)\to\zG E$ by \be\label{D}(Df|X)=a(X)(f)\;,\ee we get out of (\ref{4a}) that \be\label{4c}D(Y|Z)=Y\circ Z=[Y,Z]+[Z,Y]\;.\ee The fact that (\ref{4c}) is a consequence of the `invariance' condition (\ref{4a}) and the nondegeneracy of the scalar product, will be of importance later on. Let us moreover stress that (\ref{4c}) implies a differentiability condition for the first argument of the Leibniz bracket: \be\label{DiffCondFirst} [fX,Y]=f[X,Y]-a(Y)(f)X+(X|Y)Df\;.\ee

It is now clear that

\begin{prop} Courant algebroids $(E,[-,-],(-|-),a)$ are exactly the Leibniz algebroids $(E,[-,-],a)$ endowed with a scalar product $(-|-)$, such that, for any $X,Y,Z\in\zG E$,
\be\label{Inv1}a(X)(Y|Z)=([X,Y]|Z)+(Y|[X,Z])\;,\ee
\be\label{Inv2}a(X)(Y|Z)=(X|[Y,Z]+[Z,Y])\;.\ee
\end{prop}

As already indicated above, we view the conditions (\ref{Inv1}) and (\ref{Inv2}), as well as their consequence \be\label{Inv3}([X,Y]|Z)+(Y|[X,Z])=(X|[Y,Z]+[Z,Y])\;,\ee as the invariance properties of the scalar product. We will come back to this idea in Subsection \ref{Back}.\medskip

Finally, we define the algebraic version of Courant algebroids \cite{jub}.
\begin{defi}
A \define{Courant pseudoalgebra} over an $R$-algebra $\cA$ is an anchored module $(\cE,a)$ endowed with a Leibniz bracket $[-,-]$ and with a nondegenerate bilinear symmetric form $(-|-) \colon \cE \times \cE \to \cA$, such that
\be
([X,Y]|Y) = (X|[Y,Y])\;,\ee
\be a(X) (Y|Z) = ([X,Y]|Z) + (Y|[X,Z])\;,\ee
for any $X, Y, Z \in \cE$.
\end{defi}

As in the geometric case, these conditions imply, under the identification of $\cE$ with $\cE^*$ given by the scalar product, that
\begin{equation}
a(X) (Y|Z) = (X | Y \circ Z)\;,
\end{equation}
for any $X, Y, Z \in \cE$.\medskip

Obviously, if $E \to M$ is a Courant algebroid, then $\Gamma E$ is a Courant pseudoalgebra over $\Ci(M)$.

\section{Subclasses of Leibniz Algebroids}

\subsection{Loday Algebroids}

In \cite{GKP}, the authors consider specific Leibniz algebroids, called {\it Loday algebroids}, which have a right anchor satisfying a condition analogous to (\ref{DiffCondFirst}), and show that almost all Leibniz algebroids met in literature are Loday algebroids in their sense.

\begin{defi}\label{DLodAld} A {\em Loday algebroid} is a Leibniz algebroid $(E,[-,-],a)$ equipped with a derivation $$\textsf{D}:C^{\infty}(M)\to\Hom_{C^{\infty}(M)}(\zG (E^{\ot 2}),\zG E)\;,$$ such that
\begin{equation}\label{RightDiffProp}
[fX,Y]=f[X,Y]-a(Y)(f)X+(\textsf{D}f)(X,Y)\;,
\end{equation} for any $X, Y \in \zG E$ and $f \in \Ci(M)$.
\end{defi}

Let us mention that the right anchor $\textsf{\em D}$ can be viewed as a bundle map $\textsf{\em D}:E\to TM\otimes \op{End}E$ (whereas the left anchor is a bundle map $a:E\to TM$). Its local form is $$(\textsf{\em D}f)(X^ie_i,Y^je_j)=X^i \textsf{\em D}\,_{ij}^{\ell k} \p_k f\; Y^j e_\ell$$ (whereas the local form of $a$ is $$a(X^ie_i)(f)Y=X^i a_i^k \p_k f\; Y^\ell e_\ell\;)\;.$$

\begin{ex}[\cite{GKP}, Section 5] Leibniz algebra brackets, Courant-Dorfman brackets, twisted Courant-Dorfman brackets, Courant algebroid brackets, brackets associated to contact structures, Grassmann-Dorfman brackets, Grassmann-Dorfman brackets for Lie algebroids, Lie derivative brackets for Lie algebroids, Leibniz algebroid brackets associated to Nambu-Poisson structures... are Loday algebroid brackets.\end{ex}

For instance, it is clear from what was said above that, in the case of Courant algebroids, the derivation $\textsf{\em D}$ is given by $$\textsf{\em D}:\Ci(M)\times\zG E\times \zG E\ni (f,X,Y)\mapsto (X|Y)Df\in \zG E\;.$$

The algebraic version of Loday algebroids is defined as follows:

\begin{defi} A {\em Loday pseudoalgebra} is a Leibniz pseudoalgebra $(\cE,[-,-],a)$ over $(R,\cA)$ equipped with a derivation $$\textsf{D}:\cA\to\Hom_{\cA}(\cE\otimes_{\cA}\cE,\cE)\;,$$ such that
\begin{equation}\label{RightDiffProp}
[fX,Y]=f[X,Y]-a(Y)(f)X+(\textsf{D}f)(X,Y)\;.
\end{equation}
\end{defi}
\subsection{Symmetric Leibniz Algebroids}

We now introduce another subclass of Leibniz algebroids, {\it symmetric Leibniz algebroids}, which satisfy two weak differentiability conditions on the left argument, and contain Courant algebroids as a particular example.

\begin{defi}\label{DSymLeiAld} A {\em symmetric Leibniz algebroid} is a Leibniz algebroid $(E,[-,-],a)$ over $M$, such that, for any $f\in C^{\infty}(M)$, $X,Y\in\Gamma E$,
\be\label{S1}X\c fY-(fX)\c Y=0\;\;\text{and}\;,\ee
\be\label{S2}[fX,Y\c Z]-[X,Y]\c fZ-(fY)\c[X,Z]=0\;.\ee
\end{defi}

The definition can be equivalently formulated as follows:

\begin{prop} A {\em symmetric Leibniz algebroid} is a Leibniz algebroid $(E,[-,-],a)$ over $M$, such that, for any $f\in C^{\infty}(M)$, $X,Y\in\Gamma E$,
\be\label{S1b}X\c fY-(fX)\c Y=0\;\;\text{and}\;,\ee
\be\label{S2b}([fX,Y]-f[X,Y])\c Z + Y\c ([fX,Z]-f[X,Z])=0\;.\ee\end{prop}

\begin{proof} It suffices to show that, for a Leibniz algebroid that satisfies the first condition, the conditions (\ref{S2}) and (\ref{S2b}) are equivalent. Note first that the Jacobi identity implies that $[X,Y\circ Z]=[X,Y]\circ Z+Y\c [X,Z]$. It now follows that (\ref{S2}) is equivalent to $$[fX,Y]\c Z+Y\c [fX,Z]=[fX,Y\c Z]$$ $$=[X,Y]\c fZ+(fY)\c[X,Z]=(f[X,Y])\c Z+Y\c f[X,Z]\;.$$ \end{proof}

The first condition (\ref{S1b}) means that the symmetrized bracket is $\Ci(M)$-linear between the two variables $X,Y$. The second condition (\ref{S2b}) is a $\Ci(M)$-linearity condition in a combination of symmetrized products. \medskip

\begin{prop} A Loday algebroid $(E,[-,-],a,\textsf{D})$ over a manifold $M$ is a symmetric Leibniz algebroid if and only if, for all $f\in\Ci(M)$ and all $X,Y,Z\in\zG E,$ \be\label{LodS1}(\textsf{D}f)(X,Y)=(\textsf{D}f)(Y,X)\;\;\text{and}\;\ee \be\label{LodS2}(\textsf{D}f)(X,[Y,Z]+[Z,Y])=(\textsf{D}f)([X,Y],Z)+(\textsf{D}f)(Y,[X,Z])\;.\ee \end{prop}

\begin{proof} It follows from the differentiability properties (\ref{axiom1}) and (\ref{RightDiffProp}), and from the antisymmetry $a[X,Y]=-a[Y,X]$ of the bracket of vector fields, that the $\Ci(M)$-linearity conditions (\ref{S1}) and (\ref{S2}) are equivalent to (\ref{LodS1}) and (\ref{LodS2}).\end{proof}

\begin{ex}\label{ExSym1} Leibniz algebra brackets, Courant-Dorfman brackets, twisted Courant-Dorfman brackets, Courant algebroid brackets, brackets associated to contact structures, Grassmann-Dorfman brackets, Grassmann-Dorfman brackets for Lie algebroids... are Loday algebroid and symmetric Leibniz algebroid brackets. \end{ex}

Indeed, since for a Courant algebroid $(E,[-,-],(-|-),a)$ over $M$, with derivation $D:\Ci(M)\to \zG E$, we have $\textsf{\em D}=(-|-)D$, the $\Ci(M)$-linearity conditions (\ref{LodS1}) and (\ref{LodS2}) are direct consequences of the symmetry and the invariance properties of the scalar product. For definitions concerning the other examples, we refer the reader to \cite{GKP}, Section 5.

\begin{ex} Lie derivative brackets for Lie algebroids, Leibniz algebroid brackets associated to Nambu-Poisson structures... are Loday algebroid but $(\,$usually$\,)$ nonsymmetric Leibniz algebroid brackets.\end{ex}

We examine the first example. Let $(E,[-,-]_E,a_E)$ be a Lie algebroid, let $d^E:\zG(\w^\bullet E^*)\to \zG(\w^{\bullet +1}E^*)$ be the Lie algebroid differential, and denote by $${\cal L}^E:\zG E\times \zG(\w^\bullet E^*)\ni (X,\zw)\mapsto i_Xd^E\zw+d^Ei_X\zw\in \zG(\w^\bullet E^*)$$ the Lie algebroid Lie derivative, where $i_X$ is the interior product. There is a Leibniz bracket on sections of the vector bundle $E\oplus\w E^*$. Indeed, set, for any $X,Y\in\zG E$ and any $\zw,\zh\in\zG(\w E^*)$, \be\label{CDA1}[X+\zw,Y+\zh]=[X,Y]_{E}+{\cal L}_X^E\zh\;.\ee This is a Loday algebroid bracket with left anchor $a(X+\zw)=a_E(X)$ and right anchor $$(\textsf{\em D}h)(X+\zw,Y+\zh)=d^E h\w i_Y\zw+a_E(X)(h)\zh\;,$$ see \cite{GKP}, Section 5.\medskip

If this Loday algebroid $E\oplus \w E^*$ is symmetric, Condition (\ref{LodS1}) is satisfied in particular for 0-forms, i.e. we have $$a_E(X)(h)g=a_E(Y)(h)f\;,$$ for any $f,g,h\in\Ci(M)$ and any $X,Y\in\zG E$. If we choose $f=0$ and $g=1$, we find that $a_E=0$, so that the considered Lie algebroid $E$ is a Lie algebra bundle ({\small LAB}). Conversely, if $E$ is a {\small LAB}, we get $(\textsf{\em D}h)(X+\zw,Y+\zh)=0$, since $(d^E h)(-)=a_E(-)(h)=0$; hence, the conditions (\ref{LodS1}) and (\ref{LodS2}) are satisfied and $E\oplus \w E^*$ is a symmetric Leibniz algebroid.\medskip


Let us close this section with the definition of symmetric Leibniz pseudoalgebras:

\begin{defi} A {\em symmetric Leibniz pseudoalgebra} is a Leibniz pseudoalgebra $(\cE,[-,$ $-],a)$ over some $(R,\cA)$, such that (\ref{S1}) and (\ref{S2}), or, equivalently, (\ref{S1b}) and (\ref{S2b}), are satisfied for all $f\in\cA$ and all $X,Y,Z\in\cE$. We denote by ${\tt SymLeiPsAlg}$ $\,(R,\cA)$ $\subset$ ${\tt LeiPsAlg}\,(R,\cA)$ the full subcategory of symmetric Leibniz pseudoalgebras. \end{defi}

\begin{ex}\label{ExSym2} An example of a symmetric Leibniz pseudoalgebra is the free symmetric Leibniz pseudoalgebra over an anchored module that we describe in Theorem \ref{Main2} and in Proposition \ref{Main3}. Observe that this symmetric Leibniz bracket is not Loday.\end{ex}

\section{Generalized Courant Algebroids}

\subsection{Generalized Courant Algebroids}

As mentioned before, one of the goals of this paper is to extend the concept of Courant algebroid and to construct the free object in the category of generalized Courant algebroids over any anchored vector bundle.

\begin{defi}\label{DPrCrtAld} Let $(\cE_1,[-,-],a)$ be a Leibniz pseudoalgebra over $(R,\cA)$ and let $(\cE_2,\zm^\ell,\zm^r)$ be an $\cE_1$-module. Assume further that $(-|-):\cE_1\times\cE_1\to\cE_2$ is a symmetric $\cA$-bilinear $\cE_2$-valued map, such that, for any $X,Y,Z\in\cE_1$, the `invariance relations' \be\label{LeftAct} \zm^\ell(X)(Y|Z)=([X,Y]|Z)+(Y|[X,Z])\;,\ee \be\label{RightAct} -\zm^r(X)(Y|Z)=([Y,Z]+[Z,Y]|X)\;\;\text{and}\ee \be\label{Equal}([X,Y]|Z)+(Y|[X,Z])=([Y,Z]+[Z,Y]|X)\;\ee hold true. We refer to such a tuple $(\cE_1,\cE_2,[-,-],(-|-), a,\zm^\ell,\zm^r)$ as a {\em generalized pre-Courant pseudoalgebra}.\end{defi}

\begin{rem}\label{Natural} In the geometric situation $\cE_1=\zG E_1$, where $E_1\to M$ is a vector bundle over a manifold, $R=\R$ and $\cA=\Ci(M)$, we can take $(\cE_2,\zm^\ell,\zm^r)=(\Ci(M),a,-a)$, which is actually an $E_1$-module. If we now assume in addition that $(-|-)$ is nondegenerate, the generalized pre-Courant pseudoalgebra is a classical Courant algebroid $(E_1,[-,-],$ $(-|-),a)$.\end{rem}

\begin{prop}\label{Motiv} Any generalized pre-Courant pseudoalgebra with nondegenerate scalar product $(-|-)$ is a Loday pseudoalgebra and a symmetric Leibniz pseudoalgebra.\end{prop}

It is clear that, for any $X\in\cE_1$, we have $(X|-)\in\op{Hom}_\cA(\cE_1,\cE_2)$. By nondegenerate scalar product, we mean here that any $\zD\in\op{Hom}_\cA(\cE_1,\cE_2)$ reads $\zD=(X|-)$, and that the map $Y\mapsto (Y|-)$ is injective, so that $X$ is unique. Indeed, in the aforementioned geometric case, nondegeneracy in each fiber, see Definition \ref{CrtAld}, implies these requirements.

\begin{proof} We use the above notation; in particular $f\in\cA$ and $X,Y,Z\in\cE_1$. In view of the invariance relations and Leibniz rule for the left action, we get \be\label{R1}-\zm^r(X)(f(Y|Z))=f([Y,Z]+[Z,Y]|X)+a(X)(f)(Y|Z)\;.\ee On the other hand, the Leibniz rule for the bracket $[-,-]$ gives \be\label{R2}-\zm^r(X)(fY|Z)=([fY,Z]+[Z,fY]|X)=([fY,Z]|X)+(f[Z,Y]+a(Z)(f)Y|X)\;.\ee Note now that $a(-)(f)(Y|Z)\in\op{Hom}_\cA(\cE_1,\cE_2)$, so that there is a unique $(\textsf{\em D}f)(Y,Z)$ $\in\cE_1$ such that \be\label{R3}((\textsf{\em D}f)(Y,Z)|-)=a(-)(f)(Y|Z)\;.\ee The properties of the anchor and the scalar product imply that $\textsf{\em D}$ is a derivation $$\textsf{\em D}:\cA\to \op{Hom}_\cA(\cE_1\otimes_\cA\cE_1,\cE_1)\;.$$ It now follows from (\ref{R1}), (\ref{R2}), (\ref{R3}), and nondegeneracy that $$[fY,Z]=f[Y,Z]-a(Z)(f)Y+(\textsf{\em D}f)(Y,Z)\;,$$ i.e. that $(\cE_1,[-,-],a)$ is a Loday pseudoalgebra.\medskip

Furthermore, the latter is a symmetric Leibniz pseudoalgebra if and only if the conditions (\ref{LodS1}) and (\ref{LodS2}) are satisfied -- the proof in the geometric situation remains valid in the present algebraic case --. It is easily seen that these requirements are fulfilled due to the invariance relation (\ref{Equal}). \end{proof}

In view of Remark \ref{Natural}, it would now be natural to define generalized Courant pseudoalgebras as generalized pre-Courant pseudoalgebras with nondegenerate scalar product. Actually we choose a more general definition, see Proposition \ref{Motiv}:

\begin{defi}\label{DCrtAld} A {\em generalized Courant pseudoalgebra} is a symmetric generalized pre-Courant pseudoalgebra.\end{defi}

Generalized Courant pseudoalgebras are a full subcategory ${\tt CrtAld}$ of the category ${\tt PrCrtAld}$ of generalized pre-Courant pseudoalgebras.

\begin{defi} A {\em morphism} between two generalized (pre-)Courant pseudoalgebras $$(\cE_1,\cE_2,[-,-],(-|-),a,\zm^\ell,\zm^r)\;\;\text{and}\;\;(\cE'_1,\cE'_2,[-,-]',(-|-)',a',\zm'^\ell,\zm'^r)$$ over the same pair $(R,\cA)$, is a morphism $(\zf_1,\zf_2)$ from the $\cE_1$-module $\cE_2$ to the $\cE'_1$-module $\cE'_2$, such that \be\label{RespScProd}(-|-)'\,(\zf_1\times\zf_1)=\zf_2\,(-|-)\;.\ee\end{defi}

\subsection{Generalized Courant Algebroid Associated to a Symmetric Leibniz Algebroid}\label{Back}

The next theorem describes this Courant algebroid. It is also the motivation for the introduction of symmetric Leibniz algebroids. Moreover, it will turn out that the generalized Courant algebroid associated to a symmetric Leibniz algebroid is one of the two components of the free Courant algebroid.

\begin{theo}\label{Main} Let $(\cE,[-,-],a)$ be a symmetric Leibniz pseudoalgebra over $(R,\cA)$. Denote by $\odot$ the symmetric tensor product over $\cA$, take the subset \begin{equation}\label{Inv} \text{\em Inv}=\{[X,Y]\odot Z+Y\odot [X,Z]-X\odot([Y,Z]+[Z,Y]): X,Y,Z\in\cE\}\;\end{equation} of the $\cA$-module $\cE^{\odot 2}$, and let $\la\op{Inv}\ra$ be the $\cA$-submodule of $\cE^{\odot 2}$ generated by $\op{Inv}$. The quotient $${\cal R}(\cE)=\cE^{\odot 2}/\la\op{Inv}\ra$$ is an $\cE$-module with actions $\tilde{\zm}^\ell$ and $\tilde{\zm}^r$ induced by \be\label{LeftAct} \zm^\ell(X)(Y_1\odot Y_2)=[X,Y_1]\odot Y_2+Y_1\odot[X,Y_2]\ee and \be\label{RightAct} -\zm^r(X)(Y_1\odot Y_2)=(Y_1\c Y_2)\odot X\;.\ee These data, together with the {\em universal scalar product} \be\label{UnivScProd} (-|-): \cE\times\cE\ni (X,Y)\mapsto (X\odot Y)^{\;\widetilde{}}\in{\cal R}(\cE)\;,\ee define a generalized Courant pseudoalgebra \be\label{GenCrtAld}\cC(\cE):=(\cE,{\cal R}(\cE),[-,-],(-|-),a,\tilde{\zm}^\ell,\tilde{\zm}^r)\;.\ee\end{theo}

\begin{ex} All the examples of symmetric Leibniz brackets described in Example \ref{ExSym1} and Example \ref{ExSym2} thus give rise to generalized Courant algebroids.\end{ex}

\begin{rem} The associated generalized Courant algebroid is a very natural construction over a symmetric Leibniz pseudoalgebra, whose scalar product is the universal scalar product given by the symmetric tensor product and whose actions are the `invariant' Courant actions.\end{rem}


We first prove the following

\begin{lem} The $\cA$-module $\cE^{\odot 2}$ is an $\cE$-module for the actions $\zm^\ell$ and $\zm^r$.\end{lem}

\begin{proof} (i) We first show that $\zm^\ell(X)$ and $\zm^r(X)$ are well-defined on $\cE^{\odot 2}$ (note that we do of course not intend to show that they are $\cA$-linear on $\cE^{\odot 2}$; indeed, they are visibly only $R$-linear). Since the {\small RHS}s of (\ref{LeftAct}) and (\ref{RightAct}) are symmetric in $Y_1,Y_2$, it suffices to prove that they respect the `defining relations' of the tensor product over $\cA$. The only nonobvious condition is that $Y_{1}\odot fY_{2}=(fY_{1})\odot Y_{2}$ be preserved. And indeed, we have

\begin{eqnarray*}
\zm^\ell(X)(Y_{1}\odot fY_{2})&=&[X,Y_{1}]\odot fY_{2}+Y_{1}\odot[X,fY_{2}]\\
&=&[X,Y_{1}]\odot fY_{2}+Y_{1}\odot f[X,Y_{2}]+Y_{1}\odot a(X)(f)Y_{2}\\
&=&(f[X,Y_{1}])\odot Y_{2}+(fY_{1})\odot[X,Y_{2}]+(a(X)(f)Y_{1})\odot Y_{2}\\
&=&[X,fY_{1}]\odot Y_{2}+(fY_{1})\odot[X,Y_{2}]\\
&=&\zm^\ell(X)((fY_{1})\odot Y_{2})\;
\end{eqnarray*}
and
\begin{eqnarray*}
\zm^r(X)(Y_{1}\odot fY_{2})&=&-(Y_1\c fY_2)\odot X\\
&=&-((fY_1)\c Y_2)\odot X\\
&=&\zm^r(X)((fY_{1})\odot Y_{2})\;.
\end{eqnarray*}

(ii) It remains to check the `Leibniz morphism conditions' (\ref{VVW}), (\ref{WVV}), and (\ref{VWV}), as well as the Leibniz rule (\ref{LeibRule}). The Leibniz rule $$\zm^\ell(X)(Y_1\odot fY_2)=f\zm^\ell(X)(Y_1\odot Y_2)+a(X)(f)(Y_1\odot Y_2)$$ is clear from (i). The morphism conditions are also straightforwardly checked. To verify for instance $$\zm^r[X,Z]=\zm^r(Z)\zm^r(X)+\zm^\ell(X)\zm^r(Z)\;,$$ note first that the right adjoint action $[-,X]$ on a symmetrized product vanishes: $$[[Y_1,Y_2]+[Y_2,Y_1],X]=[Y_1,[Y_2,X]]-[Y_2,[Y_1,X]]+[Y_2,[Y_1,X]]-[Y_1,[Y_2,X]]=0\;.$$ We now get
$$\zm^r[X,Z](Y_1\odot Y_2)=-(Y_1\c Y_2)\odot [X,Z]\;,$$
\begin{eqnarray*}
\zm^r(Z)\zm^r(X)(Y_1\odot Y_2)&=& -\zm^r(Z)((Y_1\c Y_2)\odot X)\\
&=& ((Y_1\c Y_2)\c X)\odot Z\\
&=& ([Y_1\c Y_2,X]+[X,Y_1\c Y_2])\odot Z\\
&=& [X,Y_1\c Y_2]\odot Z\;,
\end{eqnarray*}
and
\begin{eqnarray*} \zm^\ell(X)\zm^r(Z)(Y_1\odot Y_2)&=&-\zm^\ell(X)((Y_1\c Y_2)\odot Z)\\
&=&-[X,Y_1\c Y_2]\odot Z-(Y_1\c Y_2)\odot [X,Z]\;.
\end{eqnarray*}
Hence, the result.
\end{proof}

The symmetric $\cA$-bilinear map
$$
\la -|-\ra:
\cE^{\times 2}\ni (X,Y) \mapsto X\odot Y\in \cE^{\odot 2}\;
$$ satisfies
\be \label{U1} \zm^\ell(X)\la Y|Z\ra=\la[X,Y]|Z\ra+\la Y|[X,Z]\ra\ee and \be\label{U2}-\zm^r(X)\la Y|Z\ra=\la X|[Y,Z]+[Z,Y]\ra\;,\ee
which are similar to (\ref{Inv1}) and (\ref{Inv2}). Since (\ref{Inv3}) does however not hold in general, we consider the quotient $\cA$-module $${\cal R}(\cE)=\cE^{\odot 2}/\la\op{Inv}\ra\;.$$

\begin{lem} The $\cA$-module ${\cal R}(\cE)$ is an $\cE$-module for the actions $\tilde{\zm}^\ell$ and $\tilde{\zm}^r$ induced by $\zm^\ell$ and $\zm^r$.\end{lem}

\begin{proof} It suffices to show that the actions descend to the quotient; indeed, the induced maps then inherit the required properties.\medskip

(i) Left action. Let $I(X,Y,Z)$, or just $I$, be any element in $\op{Inv}\subset\cE^{\odot 2}$, and let $f\in\cA$ and $W\in\cE$. Since $$\zm^\ell(W)(f I)=f\zm^\ell(W)(I)+a(W)(f)I\;,$$ we have $\zm^\ell(W)\la\op{Inv}\ra\subset\la\op{Inv}\ra,$ if $\zm^\ell(W)\op{Inv}\subset\la\op{Inv}\ra.$ The latter actually holds true:

$$ \zm^\ell(W)I(X,Y,Z) $$ $$= [W,[X,Y]]\odot Z+[X,Y]\odot[W,Z]+[W,Y]\odot[X,Z]
+Y\odot [W,[X,Z]] $$ $$
-[W,X]\odot(Y\c Z)-X\odot([W,Y]\c Z)-X\odot(Y\c[W,Z])$$
$$ =[[W,X],Y]\odot Z+[X,[W,Y]]\odot Z
+[X,Y]\odot[W,Z]+[W,Y]\odot[X,Z] $$ $$ +
Y\odot[[W,X],Z]+Y\odot[X,[W,Z]]
-[W,X]\odot(Y\c Z)-X\odot([W,Y]\c Z)-X\odot(Y\c[W,Z])$$
$$=I([W,X],Y,Z)+I(X,[W,Y],Z)+I(X,Y,[W,Z])\;.
$$

(ii) Right action. In view of the annihilation of symmetrized products by right adjoint actions and due to the symmetry condition (\ref{S2}), $$[fX,Y\c Z]=[X,Y]\c fZ+(fY)\c [X,Z]\;,$$ we get

$$\zm^r(W)(f I(X,Y,Z))$$ $$=\zm^r(W)\left([X,Y]\odot fZ+(fY)\odot [X,Z]-(fX)\odot(Y\c Z)\right)$$ $$= -([X,Y]\c fZ+(fY)\c [X,Z]-[fX,Y\c Z]-[Y\c Z,fX])\odot W$$ $$=0\;.$$ \end{proof}

It follows from (\ref{U1}) and (\ref{U2}) that $(\cE,[-,-],a)$, $({\cal R}(\cE),\tilde{\zm}^\ell,\tilde{\zm}^r)$, and the symmetric $\cA$-bilinear `universal scalar product' $$(-|-):\cE^{\times 2}\ni (X,Y)\mapsto \la X|Y\ra^{\;\widetilde{}}\in {\cal R}(\cE)\;$$ define a generalized Courant pseudoalgebra. This completes the proof of Theorem \ref{Main}.

\section{Free Symmetric Leibniz Algebroid}

The free symmetric Leibniz algebroid is the second ingredient needed for the construction of the free Courant algebroid.

\subsection{Leibniz Pseudoalgebra Ideals}

\newcommand{\cF}{{\cal F}}

\begin{defi} Let $(\cE,[-,-],a)$ be a Leibniz pseudoalgebra over $(R,\cA)$. A {\em Leibniz pseudoalgebra ideal} is an $\cA$-submodule $\cI\subset\cE$, which is a two-sided Leibniz $R$-algebra ideal, i.e. $[\cI,\cE]\subset\cI$ and $[\cE,\cI]\subset\cI$, and which is contained in the kernel of the anchor, i.e. $\cI\subset\ker a$.\end{defi}

\begin{prop} The quotient of a Leibniz pseudoalgebra by a Leibniz pseudoalgebra ideal is a Leibniz pseudoalgebra for the induced bracket and anchor.\end{prop}

\begin{proof} Obvious.\end{proof}


\begin{lem} Let $(\cE,[-,-],a)$ be a Leibniz pseudoalgebra over $(R,\cA)$, let \be\label{Ideal1} J_1=\{X\c fY-(fX)\c Y: f\in\cA,\, X,Y\in\cE\}\;\;\text{and}\ee \be\label{Ideal2} J_2=\{[fX,Y\c Z]-[X,Y]\c fZ-(fY)\c[X,Z]: f\in\cA,\, X,Y,Z\in\cE\}\;,\ee and denote by $\la J_1\ra$ $(\,$resp., $\la J_2\ra$$\,)$ the $\cA$-module generated by $J_1$ $(\,$resp., $J_2$$\,)$. The $\cA$-module $(J_1+J_2):=\la J_1\ra+\la J_2\ra$ is an ideal of the Leibniz pseudoalgebra $\cE$, so that the quotient $\cE/(J_1+J_2)$ inherits a symmetric Leibniz pseudoalgebra structure.\end{lem}


\begin{proof} The left adjoint action $[W,-]$, $W\in\cE$, satisfies the Leibniz rule with respect to the Leibniz bracket $(X,Y)\mapsto [X,Y]$, the symmetrized bracket $(X,Y)\mapsto X\c Y$, and the $\cA$-module structure $(f,X)\mapsto fX$. It follows that $$[W,X\c fY-(fX)\c Y]=$$ $$[W,X]\c fY+X\c f[W,Y]+X\c\, a(W)(f) Y-\;$$ $$(f[W,X])\c Y-\,(a(W)(f)X)\c Y-(fX)\c [W,Y]\in \la J_1\ra\;,$$ and similarly for $J_2$.

As for the right action $[-,W]$, $W\in\cE$, recall that it vanishes on every symmetrized bracket. Since the first term $[fX,Y\c Z]=[fX,Y]\c Z+Y\c[fX,Z]$ of an element of $J_2$ is symmetric as well, the sets $J_1$ and $J_2$ vanish under the right action.\medskip

For any $f\in\cA,$ $W\in\cE,$ and $Q\in J_{1},$ we have now $$Q\c fW-(fQ)\c W=[Q,fW]+[fW,Q]-[fQ,W]-[W,fQ]\in J_1\;,$$ with $[Q,fW]=0$, $[fW,Q]\in \la J_{1}\ra$, and \be\label{Ida}[W,fQ]=f[W,Q]+a(W)(f)Q\in \la J_{1}\ra\;.\ee Hence, \be\label{Idb}[fQ,W]\in\la J_1\ra\;.\ee

Similarly, if $Q\in J_{2}$, $$ W\c fQ-(fW)\c Q=[W,fQ]+[fQ,W]-[fW,Q]-[Q,fW]\in J_1\;,$$ with \be\label{Idc}[W,fQ]=f[W,Q]+a(W)(f)Q\in \la J_{2}\ra\;,\ee $[fW,Q]\in \la J_2\ra$, and $[Q,fW]=0$. Therefore, \be\label{Idd}[fQ,W]\in J_{1}+\la J_{2}\ra\;.\ee\medskip Equations (\ref{Ida}), (\ref{Idb}), (\ref{Idc}), and (\ref{Idd}) imply that the $\cA$-submodule $\la J_1\ra +\la J_2\ra\subset\cE$ is a Leibniz $R$-algebra ideal.\medskip

To see that $\la J_{1}\ra +\la J_{2}\ra$ is a Leibniz pseudoalgebra ideal, it now suffices to recall that $a$ is $\cA$-linear and that any symmetrized bracket belongs to $\ker a$.\end{proof}

\subsection{Free Leibniz Pseudoalgebra}



There exists a forgetful functor $$\op{For}:{\tt (Sym)LeiPsAlg}\,(R,\cA)\to {\tt AncMod}(\cA)\;.$$ We write for short $\cat{C} := \cat{(Sym)LeiPsAlg}\,(R,\cA)$ and $\cat{D} := \cat{AncMod}(\cA)$. Therefore, for any $\cM\in{\tt D}$, we can define the free (symmetric) Leibniz pseudoalgebra over $\cM$. It is made of an object $\op{F(S)}\cM\in {\tt C}$ and a ${\tt D}$-morphism $i:\cM\to \op{F(S)}\cM$, such that, for any object $\cE\in{\tt C}$ and any ${\tt D}$-morphism $\zf:\cM\to \cE$, there is a unique ${\tt C}$-morphism $\Phi:\op{F(S)}\cM\to \cE$, such that $\Phi\, i=\zf$ :

\begin{equation} \begin{tikzpicture}
 \matrix (m) [matrix of math nodes, row sep=3em, column sep=3em]
   { \cM & \op{F(S)}\cM  \\
       & \cE \\ };
 \path[->]
 (m-1-1) edge node[left] {$\zf$} (m-2-2)
 (m-1-1) edge node[auto] {$i$}(m-1-2)
 (m-1-2) edge node[right] {$\Phi$} (m-2-2);
\end{tikzpicture}
\end{equation}

We first recall the construction of the free Leibniz algebra over an $R$-module \cite{LP}. Let $V$ be an $R$-module and let $\overline{T} V=\bigoplus_{k\ge 1}V^{\otimes_{\!R}\,k}$ be the reduced tensor $R$-module over $V.$ The {\em universal Leibniz bracket} $[-,-]$ is defined by the requirement
$$
v_{1}\ot v_{2}\ot\cdots\ot v_{k}
=[v_{1},[v_{2},...,[v_{k-2},[v_{k-1},v_{k}]]...]]\;,
$$
$v_i\in V.$\medskip

For instance, $$[v_{1},v_{2}]=v_{1}\ot v_{2}\in V^{\otimes_{\!R}\,2}\;,$$ $$[v_1,v_2\ot v_3]=[v_1,[v_2,v_3]]=v_1\ot v_2\ot v_3\in V^{\ot_{\! R}\,3}\;,$$  $$[v_1\ot v_2,v_{3}]=[[v_1,v_2],v_3]=[v_1,[v_2,v_3]]-[v_2,[v_1,v_3]]=v_{1}\ot v_{2}\ot v_{3}-v_{2}\ot v_{1}\ot v_{3}\in V^{\ot_{\!R}\,3}\;.$$

\vspace{0.1mm}\begin{theo}\label{Main2} Let $(\cM,a)$ be an anchored $\cA$-module. The free Leibniz $(R,\cA)$-pseudo-algebra $\op{F}\cM$ over $\cM$ is the triple $(\overline{T}\cM,[-,-]_{\op{Lei}},\op{F}a)$, where $\overline{T}\cM$ is endowed with the $\cA$-module structure defined inductively by \be\label{ModStr}
f(m_{1}\ot m_{2}\ot \ldots \ot m_{n}):=
m_{1}\ot f(m_{2}\ot\ldots \ot m_{n})
-a(m_{1})(f)(m_{2}\ot\ldots \ot m_{n})\;
\ee $(f\in\cA, m_i\in\cM)$, where $[-,-]_{\op{Lei}}$ is the universal Leibniz bracket on $\overline{T}\cM$, and where $$\op{F}a:\overline{T}\cM\to \op{Der}\cA$$ is induced by $a:\cM\to\op{Der}\cA$ .
\end{theo}

\begin{proof} We denote by $\op{F}^n\cM=\cM^{\ot_{\! R}\,n}$ (resp., $\op{F}_n\cM=\bigoplus_{1\le k\le n}\cM^{\ot_{\!R}\,k}$) the grading (resp., the filtration) of $\op{F}\cM$.\medskip

\newcommand{\mf}{\mathbf{m}}

(i) Module structure. Equation (\ref{ModStr}) provides a well-defined $\cA$-module structure on $\op{F}_{n}\cM$, if we are given a well-defined $\cA$-module structure on $\op{F}_{n-1}\cM$, $n\ge 2$. Since the {\small RHS} of (\ref{ModStr}) is $R$-multilinear, the `action' is well-defined from $\op{F}^n\cM$ into $\op{F}_n\cM$. We extend it by linearity to $\op{F}_n\cM=\op{F}^n\cM\,\oplus\, \op{F}_{n-1}\cM$. It is now straightforwardly checked that this extension satisfies all the $\cA$-module requirements, except, maybe, the condition $f(g\zm)=(fg)\zm$, where $f,g\in\cA$ and $\zm\in\op{F}_n\cM$. As for the latter, note first that, if $f,g\in\cA$, $m,m_i\in\cM$, and $\mf=m_2\ot\ldots\ot m_n$, we have $$f(m\ot g\mf)=m\ot f(g\mf)-a(m)(f)(g\mf)\;,$$ since $g\mf$ is a finite sum $g\mf=\sum_{i\le n-1}\mf_i$, where $\mf_i\in\op{F}^{i}\cM$ is a decomposed tensor. Thus,
\begin{eqnarray*}
f(g(m\ot\mf))&=&f(m\ot g\mf)-f(a(m)(g)\mf)\\
&=&m\ot (fg)\mf-(a(m)(f)g)\mf-(fa(m)(g))\mf\\
&=&m\ot (fg)\mf-a(m)(fg)\mf\\
&=&(fg)(m\ot\mf)\;.
\end{eqnarray*}
\noindent The $\cA$-module structures on the filters $\op{F}_n\cM$, $n\ge 1$, naturally induce an $\cA$-module structure on $\op{F}\cM$.\medskip


(ii) Universal anchor map. Since $\op{F}\cM$ is the free Leibniz algebra over $\cM$, the map $a:\cM\to\op{Der}\cA$ factors through the inclusion $\cM\to \op{F}\cM$:
$$
\begin{CD}
\cM @>{}>> \op{F}\cM \\
@| @V{\op{F}a}VV \\
\cM@>{a}>>\op{Der}\cA
\end{CD}
$$
The Leibniz algebra morphism $\op{F}a$ is actually $\cA$-linear. Indeed, in view of the decomposition $f\mf=\sum_{i\le n-1}\mf_i$, where $\mf_i\in\op{F}^{i}\cM$ is a decomposed tensor, we obtain first $m\ot f\mf=[m,f\mf]$, where the notation is the same as above. Since, by induction, $\op{F}a$ is $\cA$-linear on $\op{F}^{n-1}\cM$, we then have
\begin{eqnarray*}
\op{F}a(f(m\ot \mf))&=&\op{F}a(m\ot f\mf)-\op{F}a(a(m)(f)\mf)\\
&=&\op{F}a[m,f\mf]-a(m)(f)\op{F}a(\mf)\\
&=&[a(m),f \op{F}a(\mf)]-a(m)(f)\op{F}a(\mf)\\
&=&f[a(m),\op{F}a(\mf)]\\
&=&f\op{F}a(m\ot \mf)\;.
\end{eqnarray*}

(iii) Leibniz pseudoalgebra conditions. To see that $(\op{F}\cM,[-,-]_{\op{Lei}},\op{F}a)$ (in the following we omit subscript $\op{Lei}$) is a Leibniz pseudoalgebra, it now suffices to check that (\ref{axiom1}) is satisfied. We proceed by induction and assume that, for any $f\in\cA$, $\mf\in \op{F}_{n-1}\cM$ and $\mf'\in \op{F}\cM$, $n\ge 2$, the bracket $[\mf,f\mf']$ satisfies Condition (\ref{axiom1}). Indeed, for $n=2$, we have $$[m,f\mf']=m\ot f\mf'=f(m\ot \mf')+a(m)(f)\mf'=f[m,\mf']+\op{F}a(m)(f)\mf'\;.$$ It is easily seen that (\ref{axiom1}) is then also satisfied in $\op{F}_n\cM$:
\bea
[m\ot\mf, f\mf']
&=&[[m,\mf],f\mf']\\
&=&[m,[\mf,f\mf']]-[\mf,[m,f\mf']]\\
&=&[m,f[\mf,\mf']]+[m,\op{F}a(\mf)(f)\mf']
-[\mf,f[m,\mf']]-[\mf,a(m)(f)\mf']\\
&=&\cdots\\
&=&f[m,[\mf,\mf']]-f[\mf,[m,\mf']]
+[\op{F}a(m),\op{F}a(\mf)](f)\mf'\\
&=&f[m\ot \mf,\mf']+\op{F}a(m\ot\mf)(f)\mf'.
\eea

(iv) Freeness. It remains to show that the Leibniz $(R,\cA)$-pseudoalgebra $(\op{F}\cM,[-,$ $-],\op{F}a)$ (we omit $\op{Lei}$), together with the anchored $\cA$-module morphism $i:\cM\hookrightarrow \op{F}\cM$, is free. Let thus $(\cE,[-,-]',a')$ be any Leibniz $(R,\cA)$-pseudoalgebra and let $\zf:\cM\to \cE$ be any anchored $\cA$-module morphism. Since $(\op{F}\cM,[-,-],i)$ is the free Leibniz $R$-algebra over the $R$-module $\cM$, the $R$-linear map $\zf$ extends uniquely to a Leibniz $R$-algebra map $\op{F}\zf:\op{F}\cM\to\cE$. When assuming that $a'\op{F}\zf=\op{F}a$ (resp., that $\op{F}\zf$ is $\cA$-linear) on $\op{F}_{n-1}\cM$, the usual proof based on the observation that $m\ot f\mf=[m,f\mf]$ (resp., on this observation combined with (\ref{ModStr}) and (\ref{axiom1})) allows to see that the same property holds on $F_n\cM$.
\end{proof}

\subsection{Free Symmetric Leibniz Pseudoalgebra}

\begin{proposition}\label{Main3} Let $J_1$, $J_2$ be the ideals (\ref{Ideal1}) and (\ref{Ideal2}) associated to the free Leibniz $(R,\cA)$-pseudoalgebra $(\op{F}\cM,[-,-]_{\op{Lei}},\op{F}a)$ over an anchored $\cA$-module $(\cM,a)$. The quotient $\op{FS}\cM:=\op{F}\cM/(J_{1}+J_{2})$, with induced bracket, anchor and `inclusion', is the free symmetric Leibniz pseudoalgebra over the anchored module $\cM$.\end{proposition}

\begin{rem} The free symmetric Leibniz pseudoalgebra over an anchored module $(\cM,a)$ is the natural quotient of the free Leibniz pseudoalgebra over $(\cM,a)$. The latter is the reduced tensor $R$-module $\op{F}\cM=\overline{T}\cM$ over $\cM$, endowed with an $\cA$-module structure that encodes the anchor $a$, the universal Leibniz bracket $[-,-]_{\op{Lei}}$, and the induced anchor $\op{F}a$. In the geometric case, when $\cM=\zG(M)$, with $M\to B$ an anchored vector bundle over a manifold, the module $\overline{T}\cM$ is not a space of sections, since the tensor product in $\overline{T}\cM$ is over $R$.\end{rem}

\begin{proof} We characterize the classes and the mentioned induced data by the symbol `tilde'. It has already been said that $(\op{FS}\cM,[-,-]^{\,\widetilde{}}_{\op{Lei}},(\op{F}a)^{\;\widetilde{}}\;)$ is a symmetric Leibniz pseudoalgebra (as usual we will omit $\op{Lei}$). On the other hand, it is clear from the definition of $(\op{F}a)^{\;\widetilde{}}$ that $\tilde \imath:\cM\ni m\to \tilde m\in\op{FS}\cM$ is an anchored $\cA$-module map.

As for freeness, let $(\cE,[-,-]',a')$ be a symmetric Leibniz pseudoalgebra. Any anchored module map $\zf:\cM\to \cE$ uniquely extends to a Leibniz pseudoalgebra map $\op{F}\zf:\op{F}\cM\to\cE$. To see that $\op{F}\zf$ descends to $\op{FS}\cM$, it suffices to show that it vanishes on $J_1$ and $J_2$. Observe first that, for any $\zm,\zn\in\op{F}\cM$, we have $$\op{F}\zf(\zm\circ\zn)=\op{F}\zf(\zm)\circ'\op{F}\zf(\zn)\;.$$ It follows now from the $\cA$-linearity of $\op{F}\zf$ and the symmetry of $\cE$ that $\op{F}\zf$ annihilates $J_1$ and $J_2$. It is also straightforwardly checked that the induced map $(\op{F}\zf)^{\;\widetilde{}}:\op{FS}\cM\to\cE$ is a map of Leibniz pseudoalgebras such that $(\op{F}\zf)^{\;\widetilde{}}\;\tilde\imath=\zf$. As for uniqueness of this extension, note that any Leibniz pseudoalgebra morphism $\op{FS}\zf:\op{FS}\cM\to\cE$ that extends $\zf$, implements a Leibniz pseudoalgebra morphism $$(\op{FS}\zf)^{\;\bar{}}:\op{F}\cM\ni \zm\mapsto \op{FS}\zf(\tilde{\zm})\in\cE$$ that extends $\zf$; hence, $(\op{FS}\zf)^{\;\bar{}}=\op{F}\zf$ and $\op{FS}\zf=(\op{F}\zf)^{\;\widetilde{}}$.\end{proof}


\section{Free Courant Algebroid}\label{Repr}

There is a forgetful functor $\op{For}:{\tt CrtAld}\to {\tt AncMod}$ between the categories of generalized Courant $(R,\cA)$-pseudoalgebras and anchored $\cA$-modules.

\begin{theo}\label{Main4} The free Courant algebroid over an anchored module $(\cM,a)$ is the generalized Courant pseudoalgebra $${\cal C}(\op{FS}\cM)=(\op{FS}\cM,{\cal R}(\op{FS}\cM),[-,-]^{\;\widetilde{}}_{\op{Lei}},(-|-),(\op{F}a)^{\;\widetilde{}},\tilde{\zm}^\ell,\tilde{\zm}^r)\;,$$ associated to the free symmetric Leibniz pseudoalgebra over $\cM$, together with the anchored module map $\tilde{\imath}:\cM\to \op{FS}\cM$. More precisely, for any generalized
Courant pseudoalgebra $${\cal C}=(\cE_1,\cE_2,[-,-]',(-|-)',a',\zm'^\ell,\zm'^r)$$ and any anchored module map $\zf:\cM\to\cE_1$, there exists a unique morphism of generalized Courant pseudoalgebras $(\zf_1,\zf_2)$ from ${\cal C}(\op{FS}\cM)$ to $\cal C$, such that $\zf_1\,\tilde{\imath}=\zf$.\end{theo}

\newcommand{\tm}{\tilde{\zm}}
\newcommand{\tn}{\tilde{\zn}}
\newcommand{\tta}{\tilde{\zt}}

\begin{proof} Let $\cal C$ and $\zf$ be as in the statement of the theorem. Due to freeness of $\op{FS}\cM$, the anchored module map $\zf$ uniquely factors through $\op{FS}\cM$, thus leading to a unique Leibniz pseudoalgebra morphism $\zf_1:\op{FS}\cM\to \cE_1$ such that $\zf_1\,\tilde{\imath}=\zf$. As for $\zf_2$, note that the $\cA$-linear map $$\zf_2=(-|-)'(\zf_1\odot \zf_1):\op{FS}\cM\odot \op{FS}\cM\to \cE_1\odot \cE_1\to \cE_2$$ descends to the quotient ${\cal R}(\op{FS}\cM)$. Indeed, if $I(\tilde\zm,\tilde\zn,\tilde\zt)\in\op{Inv}$ (in the following we omit the symbol `tilde'), we have $$\zf_2(I(\zm,\zn,\zt))=$$ $$([\zf_1\zm,\zf_1\zn]'|\zf_1\zt)'+(\zf_1\zn|[\zf_1\zm,\zf_1\zt]')'-(\zf_1\zm|[\zf_1\zn,\zf_1\zt]'+[\zf_1\zt,\zf_1\zn]')'=0\;,$$ since $\cal C$ is a generalized Courant pseudoalgebra. The resulting $\cA$-linear map ${\cal R}(\op{FS}\cM)\to \cE_2$ will still be denoted by $\zf_2$. Since \be\label{UnivScProdS}(\zm|\zn)=(\zm\odot\zn)^{\;\widetilde{}}\;,\ee it is clear that the requirement (\ref{RespScProd}), i.e. $$(-|-)'(\zf_1\times\zf_1)=\zf_2(-|-)\;,$$ is satisfied. It now suffices to check that the conditions (\ref{RespActs}), i.e. $$\zm'^\ell(\zf_1\times\zf_2)=\zf_2\,\tilde\zm^\ell\;\;\text{and}\;\;\zm'^r(\zf_2\times\zf_1)=\zf_2\,\tilde\zm^r\;,$$ hold as well. Let us detail the last case. Let $\tm,\tn,\tta\in \op{FS}\cM$ (and omit again the `tilde'). Since $$\tilde\zm^r(\zm)(\zn\odot\zt)^{\;\widetilde{}}=(\zm^r(\zm)(\zn\odot\zt))^{\;\widetilde{}}=(-(\zn\c \zt)\odot\zm)^{\;\widetilde{}}\;,$$ the application of $\zf_2$ leads to $$-(\zf_1\zn\c' \zf_1\zt|\zf_1\zm)'\;.$$ The latter coincides with the value of $\zm'^r(\zf_2\times\zf_1)$ on the arguments $(\zn\odot\zt)^{\;\widetilde{}}$ and $\zm$. Finally, uniqueness of $\zf_1$ was already mentioned, and uniqueness of $\zf_2$ is a consequence of (\ref{RespScProd}) and (\ref{UnivScProdS}).
\end{proof}

\section{Acknowledgements} 

Beno\^it Jubin thanks the Luxembourgian National Research Fund for support via AFR grant PDR 2012-1, 3963765. The research of N. Poncin was supported by Grant GeoAlgPhys 2011-2014 awarded by the University of Luxembourg. Kyousuke Uchino is grateful for an invitation to the University of Luxembourg which was the starting point of the present joint work. The authors would also like to thank Yannick Voglaire for useful discussions.

\vskip1cm
\noindent Beno\^it JUBIN\\University of Luxembourg\\Email: benoit.jubin@uni.lu\\ Norbert PONCIN\\University of
Luxembourg
\\ Email: norbert.poncin@uni.lu \\
\noindent Kyousuke UCHINO\\
Email: kuchinon@gmail.com\\

\end{document}